\documentclass[12pt,a4paper]{amsart}
\usepackage{cmap} 
\usepackage[T2A,T1]{fontenc} 

\usepackage{amssymb}
\usepackage{amsthm}
\usepackage{amsmath}
\usepackage{mathrsfs}
\usepackage[pdftex,unicode]{hyperref}
\usepackage[noadjust]{cite}

\usepackage[T2A,T1]{fontenc}

\usepackage{stackengine}
\stackMath
\def\cl#1{\stackengine{0pt}{}{\overline{\phantom{#1}}}{O}{l}{F}{T}{L}#1}

\def\es{\varnothing}
\def\om{\omega}
\def\al{\alpha}
\def\Q{\mathbb{Q}}
\def\R{\mathbb{R}}
\def\norm{\lVert\boldsymbol\cdot\rVert}

\def\A{\mathcal{A}}
\def\B{\mathcal{B}}

\def\Int{\operatorname{Int}}
\def\cl#1{\overline{#1}}
\def\supp{\operatorname{supp}}

\newtheorem{theorem}{Theorem}[section]
\newtheorem{lemma}[theorem]{Lemma}

\newtheorem{proposition}[theorem]{Proposition}
\newtheorem{corollary}[theorem]{Corollary}

\theoremstyle{definition}
\newtheorem{definition}[theorem]{Definition}

\newtheorem{question}{Question}

\theoremstyle{remark}
\newtheorem{remark}[theorem]{Remark}
\newtheorem{example}[theorem]{Example}

\begin{document}

{
\renewcommand*{\thefootnote}{$\star$}

\title[Metrizable and Fr\'{e}chet--Urysohn subgroups of TVS]{Metrizable and Fr\'{e}chet--Urysohn\\ subgroups and linear subspaces\\ of topological vector spaces}
\footnotetext[0]{The work of O.~Sipacheva was financially supported by the Russian Science Foundation, grant~22-11-00075-P.}

}

\author[Arkady Leiderman, Evgenii Reznichenko, and Ol'ga Sipacheva]{Arkady Leiderman, Evgenii Reznichenko, and Ol'ga Sipacheva}
\address{Department of Mathematics, Ben-Gurion University of the Negev, 
Beer Sheva, P.O.B. 653, Israel}
\email{arkady@math.bgu.ac.il}  
\address{Department of General Topology and Geometry, Faculty of Mechanics and Mathematics, M.~V.~Lomonosov Moscow State University, Leninskie Gory 1, Moscow, 119991, Russia}
\email{erezn@inbox.ru}
\address{Department of General Topology and Geometry, Faculty of Mechanics and Mathematics, M.~V.~Lomonosov Moscow State University, Leninskie Gory 1, Moscow, 119991, Russia}
\email{ovsipa@gmail.com}



\begin{abstract}
For a normed space $E$, we denote by $E_w$ and $E'_{w^*}$ the space $E$ and its dual $E'$ endowed with the weak 
topology $w$ and the weak$^*$ topology $w^*$, respectively. It is well known that for a Banach space $E$, both 
spaces $E_w$ and $E'_{w^*}$ are Fr\'echet--Urysohn topological spaces if and only if $E$ is finite-dimensional.

This statement remains true for all linear subspaces of $E_w$ and $E'_{w^*}$. Moreover, for any normed space $E$, 
every Fr\'echet--Urysohn subgroup of $E_w$ is locally finite-dimensional; that is, it has an open subgroup 
contained in a finite-dimensional linear subspace. Similarly, for any Banach space $E$, every Fr\'echet--Urysohn 
subgroup of $E'_{w^*}$ is locally finite-dimensional.

Analogous results are obtained for the free locally convex space $L(X)$, its weak version $L_p(X)$, and the free 
topological vector space $V(X)$ over a Tychonoff space~$X$. The same conclusions hold for Baire subgroups and 
Baire linear subspaces of $E_w$, $E'_{w^*}$, $L(X)$, $L_p(X)$, and $V(X)$.

We emphasize that even for metrizable subgroups, these results are new and nontrivial.
In fact, all results are established for the much wider family of $\varkappa$-Fr\'echet--Urysohn subgroups of  topological vector spaces. 
 The class of $\varkappa$-Fr\'echet--Urysohn spaces is the natural framework in which the proofs of our results work.
\end{abstract}

\maketitle

\section{Introduction}\label{Introduction}
The core concept of a metric space was introduced by M. Fr\'echet in 1906, 
the term ``metric space'' was coined by F. Hausdorff in 1914, 
and the modern framework for a metrizable space---a topological space whose topology can be generated by a metric---was formally characterized through metrization theorems by P. S. Urysohn in 1925. 

The topology of a metric space is completely described by convergent sequences.
However, sequences do not suffice to describe the closure of a set in a general topological space.
One of the most natural classes of topological spaces where sequences do suffice
is the class of Fr\'echet--Urysohn spaces.

\begin{itemize}
\item A topological space $X$ is  said to be 
{\em Fr\'{e}chet--Urysohn} if, for any set $A\subset X$ and any point $x$ in the closure $\cl{A}$, there exists a sequence $(x_n)_{n\in \omega} \subset A$
 converging to~$x$. 
\end{itemize}

A topological space $X$ is called a {\em $k$-space} if $A\subset X$ is closed whenever 
$A\cap K$ is closed in $K$ for every compact subset $K\subset X$.
A space $X$ is Fr\'echet--Urysohn if and only if every subspace of $X$ is a $k$-space  \cite{A1}.
The formal concept of a Fr\'echet--Urysohn topological space is due to
A.~V.~Arhangel'skii; see \cite{A1} and~\cite{A2}. 

Every first-countable space, and in particular every metrizable space, is Fr\'echet--Urysohn. 
The larger class of sequential spaces, in which closed sets are determined by convergent sequences, was 
introduced earlier by S.~P.~Franklin \cite{Franklin1965}. The good behavior of spaces whose topology is described 
by convergent sequences is one of the reasons why metrizable spaces play such a prominent role in analysis.

In functional analysis,  the transition from finite-dimensional spaces to 
in\-fi\-nite-dimensional  ones
frequently requires weakening the topology to guarantee compactness or  establish the weak convergence of 
bounded sequences. Note that already in the Hilbert space $\ell_2$, zero belongs to the weak closure of the set 
$\{e_m+m e_n: 1\le m<n\in \omega\}$, but no sequence of elements of this set converges weakly to zero (J.~von 
Neumann~\cite{vonNeumann1930}).

For a normed space $E$, we denote by $E_w$ and $E'_{w^*}$ the space $E$ and its dual $E'$ endowed with the weak topology $w$ and the weak$^*$ topology $w^*$, respectively.

 It is a folklore result  that if $E$ is an infinite-dimensional normed space, 
then $E_w$ is not metrizable \cite[p. 10]{Diestel}. 
If $E$ is an infinite-dimensional Banach space, then neither $E_w$ nor $E'_{w^*}$ is a $k$-space; for an elementary proof of these facts, see \cite[Proposition~6.4]{KM2017}.

There are more publications with similar results for other topological properties. We mention here just a few 
examples. If $E$ is a Banach space, then $E_w$ has the Pytkeev property if and only if $E$ is finite-dimensional 
(\cite[Theorem 1.4]{GKZ2017}). Moreover, it  was observed in \cite{Edgar1984} that $E_w$ is 
not a Baire space for any infinite-dimensional normed space $E$  (recall that a topological 
space $X$ is \emph{Baire} if the Baire category theorem holds for $X$, i.e., the intersection of any countable family of dense open subsets of $X$ is dense in $X$). 

On the other hand, weak topologies have many large convex compact metrizable subspaces; for example, all 
closed balls of the dual of a separable Banach space are metrizable and compact in the weak$^*$ topology. In 
particular, a closed ball of any separable reflexive space is such a space in the weak topology. Recall also that 
any Eberlein compactum, i.e., a weakly compact subset of a Banach space, is Fr\'echet--Urysohn by the 
Eberlein--\v Smulian theorem; see \cite{Pryce}. 

It is worth noting that, in general,  properties of the closed unit ball $B$ of a Banach space $E$ 
endowed with the weak topology can differ dramatically from  those of $E_w$. Indeed, 
the following conditions on a Banach space $E$ are equivalent: 
(a)~$B_w$ is a Fr\'echet--Urysohn space; (b)~$B_w$ is a sequential space; (c)~$B_w$ is a $k$-space; (d)~$E$ 
contains no isomorphic copy of~$\ell_1$ (see \cite[Theorem~5.1]{SW1991}). 
 
In what follows, the abbreviation TVS stands for a topological vector space. 
In this paper, we are interested in the subgroups and linear subspaces of a given TVS. 
Even a countable topological group does not have to be metrizable; for example,  
the group of integers  with the Bohr topology  is nonmetrizable.
By the Birkhoff--Kakutani theorem, a topological group is metrizable  if and only if it is 
first-countable. We note that Malykhin's long-standing question of whether every countable Fr\'echet--Urysohn 
group is metrizable has turned out to be independent of ZFC~\cite{HrusakRamosGarcia2014}.
  
We pose the following natural questions.
\begin{question}
\label{q_main1}
\begin{enumerate}
\item[(1)] Which TVS have the property that every Fr\'echet--Ury\-sohn (or metrizable) linear subspace is finite-dimensional?
\item[(2)] Which TVS have the property that every Baire linear subspace is finite-dimensional?
\end{enumerate}
\end{question}
 We also ask similar questions for {\em subgroups} of TVS.
 
Particular cases of Question \ref{q_main1} were considered in \cite{KL2026}.
For linear subspaces, systematic research has been undertaken very recently in \cite{MR-fu-tvp},
where the classes $\mathscr D_M$ and 
$\mathscr D_{FU}$ of TVS all of whose metrizable (respectively, Fr\'echet--Urysohn) linear 
subspaces are finite-dimensional were introduced. 
It was shown there that the relation between these classes cannot 
be decided in ZFC ($\mathscr D_M=\mathscr D_{FU}$ in the Hru\v{s}\'ak--Ramos-Garc\'{\i}a model, while 
$\mathscr D_M\ne\mathscr D_{FU}$ under $\mathrm{MA}+\neg\mathrm{CH}$) and that $V_w\in \mathscr D_{FU}$ for every 
locally convex space $V$ admitting a continuous norm; the proofs  are measure-theoretic.

In the present paper, 
we use a different, topological-algebraic approach.
In fact, all results are established for the much wider family of $\varkappa$-Fr\'echet--Urysohn subgroups of TVS. 

A topological space $X$ is  said to be
{\em $\varkappa$-Fr\'{e}chet--Urysohn} if, for any {\em open} set $A\subset X$ and any point $x$ in the closure $\cl{A}$, there exists a sequence $(x_n)_{n\in \omega} \subset A$
 converging to~$x$. This notion was introduced by A.~V.~Arhangel'skii in 1999
 and  has found numerous applications in topological algebra and functional analysis 
\cite{G2019-Baire_functions, GOR2026kFUgr, GOR-Baire, GR2026New_classes, KrupskiKucharskiMarciszewski2025, LiL, Osipov, Sakai1, Sakai2}.

For several  important TVS $E$, we establish that
every Fr\'echet--Urysohn subgroup of $E$ is \emph{locally finite-dimensional}, i.e., it has an open subgroup contained in a 
finite-dimensional linear subspace. For subgroups, this is  a natural analogue of finite-dimensionality: 
every discrete subgroup is metrizable and locally finite-dimensional, but it is not generally contained in 
a finite-dimensional linear subspace.

Evidently,  any 
Fr\'echet--Urysohn space is $\varkappa$-Fr\'echet--Urysohn, so all results stated below for 
$\varkappa$-Fr\'echet--Urysohn subgroups apply to metrizable and Fr\'e\-chet--Urysohn subgroups. We also prove the 
same conclusions for Baire subgroups.

 The main results  of the present paper  
can be summarized as follows. 
\begin{enumerate} 
\item[(1)] 
For any normed space $E$, every $\varkappa$-Fr\'echet--Urysohn or Baire 
subgroup of $E_w$ is locally finite-dimensional, and every  $\varkappa$-Fr\'echet--Urysohn or Baire 
linear subspace of $E_w$ is finite-dimensional 
(Theorem~\ref{th:bs-kfu-baire-sgs} and Corollary~\ref{cor:bs-kfu-baire-sgs}). 
\item[(2)] 
For any Banach space 
$E$, every $\varkappa$-Fr\'echet--Urysohn or Baire subgroup of $E'_{w^*}$ is locally finite-dimensional, and 
every  $\varkappa$-Fr\'echet--Urysohn or Baire linear subspace of $E'_{w^*}$ is finite-dimensional (Theorem~\ref{th:weakstar}). The completeness of $E$ is 
essential (Example~\ref{ex:incomplete}). 
\item[(3)] 
 Let $X$ be any Tychonoff space, and let $E(X)$ be the vector space with Hamel basis $X$ endowed with 
the topology of the free locally convex space $L(X)$, of its weak version $L_p(X)$, or of the free TVS $V(X)$, 
with any vector space topology between those of $L_p(X)$ and $L(X)$, or with the weak topology 
$\sigma(E(X),C_b(X))$ generated by the bounded continuous functions on~$X$ (the space $L_b(X)$ in 
Section~\ref{Preliminaries}). Then every 
$\varkappa$-Fr\'echet--Urysohn subgroup of  $E(X)$ is locally finite-dimensional, and every 
 $\varkappa$-Fr\'echet--Urysohn linear subspace of $E(X)$ is 
finite-dimensional (Theorems~\ref{t4} and~\ref{th:VX-kfu}, Corollary~\ref{cor:lbx-kfu}, and Theorem~\ref{th:lss-kfu}). 
Theorem~\ref{th:lss-kfu}  solves Problem~6 of~\cite{KL2026}. 
\item[(4)] 
 Let $X$ be any Tychonoff space, and let $E(X)$ be endowed with any vector space topology finer than that of $L_b(X)$ (in particular, with any 
of the topologies listed in~(3)). Then every
Baire subgroup of  $E(X)$ is locally finite-dimensional, and every  Baire linear subspace of 
$E(X)$ is finite-dimensional  (Theorems~\ref{th:bsg-lfd}  and~\ref{th:B} and Corollary~\ref{cor:B}).
\end{enumerate}

We emphasize that the results concerning subgroups are new and nontrivial even for metrizable subgroups of 
$E_w$ and $E'_{w^*}$.  The statements 
about Fr\'echet--Urysohn and Baire linear subspaces of $E_w$ and $E'_{w^*}$ in  (1) and (2) 
 can be derived, e.g., from 
\cite{KrupskiKucharskiMarciszewski2025, KM2017}, although they have not been formulated in this exact form 
in the literature. The results  in (3) and (4) are entirely 
novel.

The main tool for  obtaining these results is  
Theorem~\ref{th:criterion}, which says that if a TVS 
admits a norm whose balls are closed and precompact and absorb  convergent sequences, then all 
of its $\varkappa$-Fr\'echet--Urysohn and Baire subgroups are locally finite-dimensional. Along the way, we 
characterize locally finite-dimensional subgroups as locally precompact ones (Theorem~\ref{th:lfd-subgroups}).

Results on $\varkappa$-Fr\'echet--Urysohn and Baire subgroups of free topological groups obtained by similar 
methods and  detailed proofs of some auxiliary facts on $\varkappa$-Fr\'echet--Urysohn spaces 
used  in what follows are given  
in the companion paper~\cite{LRS-groups}.

\section{Preliminaries}\label{Preliminaries}

All topological spaces under consideration are assumed to be Tychonoff. 
In what follows, we sometimes identify sequences with the indexed sets of their values. 

As usual, $\R$, $\Q$, and $\mathbb Z$ denote the sets of real, rational, and integer numbers, respectively. 
 The set of nonnegative integers is 
denoted by $\omega$. Given a subset $A$ of a group $G$  (of a TVS $E$), by  $\langle A\rangle$ we denote the smallest subgroup of $G$ (respectively, of the additive 
group of $E$) that contains~$A$. We denote the 
cardinality of a set $X$ by $|X|$ and the closure and interior of a subset $X$ of a topological space by $\cl X$ 
and $\Int X$, respectively.

All TVS are assumed to be over the field $\R$ and Hausdorff. For a TVS $E$, $E'$ denotes 
the space of continuous linear functionals on $E$ and $E_w$ denotes $E$ endowed with the weak topology 
$\sigma(E,E')$. 

We say that a space $Z$ is the \emph{inductive limit} of an increasing sequence $(K_n)_{n\in\omega}$ of its 
subspaces if $Z=\bigcup_{n\in\omega}K_n$ and a set $F\subset Z$ is closed in $Z$ if and only if $F\cap K_n$ is 
closed in $K_n$ for every $n\in\omega$.

The \emph{Ra\u\i kov completion} $\varrho G$ of a topological group $G$ is the completion of $G$ with respect to 
its two-sided uniformity \cite[Section~3.6]{AT}; if $G$ is a dense subgroup of a Ra\u\i kov complete group $H$, 
then $H=\varrho G$. For an abelian group, and in particular for the additive group of a TVS 
$E$, $\varrho E$ is the usual completion~$\widehat E$ of~$E$, which is again a TVS.

Let $X$ be a set.
We denote by $E(X)$ the vector space in which $X$ is a Hamel basis.
Every nonzero $\chi\in E(X)$ can be represented uniquely (up to the order of terms) as 
$\chi=\sum_{i=1}^n\lambda_i x_i$, where  $n\in \omega\setminus\{0\}$, 
$\lambda_i\in \mathbb R\setminus \{0\}$,  $x_i\in X$, 
and $x_i\neq x_j$ for distinct $i,j\leq n$. We define 
\[
\supp\chi=\{x_1,\dots,x_n\},\quad
\|\chi\|=\sum_{i=1}^n |\lambda_i|,\quad\text{and}\quad
\chi(f)=\sum_{i=1}^n \lambda_i f(x_i)
\]
for $f\in \R^X$; we also set $\supp 0=\es$. The function $\norm$ is a norm on $E(X)$, and $\chi(\cdot)$ is 
a linear functional on~$\R^X$.

For $A\subset E(X)$, we set
$$
\supp (A) = \bigcup\{\supp (\chi) : \chi\in A\}.
$$
We also use the notation 
\[
E_n(X)=\{\chi\in E(X): |\supp\chi|\leq n\}, \quad
B_n(X)=\{\chi\in E(X): \|\chi\|\leq n\}
\]
for $n\in \omega$. 
For $A\subset X$, we identify $E(A)$ with the linear subspace of $E(X)$ spanned by $A$; thus, 
$E(A)=\{\chi\in E(X):\supp\chi\subset A\}$, and every finite-dimensional linear subspace of $E(X)$ is 
contained in $E(M)$ for some finite subset $M$ of $X$.

In this paper, we consider the following TVS generated by a Tychonoff space~$X$; 
as vector spaces, all of them coincide with $E(X)$:
\begin{itemize}
\item
$V(X)$, the free TVS of $X$; 
\item
$L(X)$, the free locally convex space of $X$; 
\item
$L_p(X)$, the space $L(X)$ endowed with the weak topology; since the dual of $L(X)$ is $C(X)$, this is 
the topology on $E(X)$ generated by the functionals $\chi\mapsto\chi(f)$, $f\in C(X)$;
\item
$L_b(X)$, the space $E(X)$ with the topology generated by the functionals $\chi\mapsto\chi(f)$, 
$f\in C_b(X)$, where $C_b(X)$ is the space of bounded continuous functions on $X$. 
\end{itemize}

Free  locally convex spaces were introduced by  A.~A.~Markov \cite{Mar, Mar1}.
 Later D.~A.~Raikov constructed the free locally convex space $L(X)$ for a uniform space $X$ \cite{Raikov}.
Free locally convex spaces constitute a very important subclass of locally convex spaces.
It suffices to mention that every locally convex space is a linear continuous quotient of a free locally convex 
space and that every Tychonoff space $X$ embeds as a closed subspace into $L(X)$. Free TVS were apparently 
first introduced by Flood~\cite{Flood}.

 The topology of the free locally convex space $L(X)$ (of the free TVS $V(X)$) 
is defined by the  property  that 
 every continuous map $f$ from $X$ to a locally convex space $E$ (respectively, to a TVS 
$E$) can be extended to a unique continuous linear operator
 $\widehat f\colon L(X) \to E$ (respectively, $\widehat f\colon V(X)\to E$).  

The free locally convex space $L(X)$ and the free TVS $V(X)$ always exist and are unique.
The topology of $V(X)$ is finer than the topology of $L(X)$, which is finer than the topology of $L_p(X)$, which, 
in turn, is finer than the topology of $L_b(X)$. Note that since every $f\in C_b(X)$ extends to a continuous 
function on the Stone--\v{C}ech compactification $\beta X$, it follows that $L_b(X)$ is a linear subspace of $L_p(\beta X)$ with 
the induced topology.

 The spaces $L(X)$ and $V(X)$ are Baire if and only if $X$ is finite.
The same is true for metrizability:  $L(X)$ and $V(X)$ are metrizable if and only if $X$ is finite.
For more details about the topological properties of $V(X)$ and $L(X)$ we refer the reader to \cite{GM2017ftvs} 
and \cite{Uspenskii}, respectively.
An intrinsic description of the topologies of both $V(X)$ and $L(X)$ is given in~\cite{BL2018}.

Recall that a family $\{M_\al:\al\in A\}$ of subsets of a space $X$ is \emph{discrete} if 
any point $x\in X$ has a neighborhood $U$ such that $|\{\al\in A:M_\al\cap U\neq\es\}|\leq 1$.

A subset $A$ of a space $X$ is said to be \emph{functionally bounded} if its image $f(A)$
under any continuous function $f\colon X\to \R$ is bounded in $\R$. 
It is easy to show (e.g., by analogy with \cite[Theorem 3.10.21]{Engelking}) 
that a set $A\subset X$ is functionally bounded if and only if, 
for any countable discrete family $\{V_n: n\in\om\}$ of open subsets of $X$, 
the set $\{n\in\om: V_n\cap A\neq \es\}$ is finite. Clearly, any compact subset of $X$ is functionally bounded.

A subset $A$ of a topological group $G$ is called \emph{precompact} or \emph{totally bounded} in $G$ 
if, for every neighborhood $O$ of the identity element of $G$, one can find a finite set $M\subset G$ 
such that $A\subset MO$. Any functionally bounded subset of a topological group is precompact 
\cite[Proposition 6.10.2]{AT}.
A subset $A$ of a topological group $G$ is precompact in $G$ if and only if the closure of $A$ 
in the Ra\u\i kov completion of $G$ is compact \cite[Theorem 3.7.10]{AT}. 
Since the Ra\u\i kov completion of a subgroup $H$ of $G$ is the closure of $H$ in $\varrho G$, it follows 
that \emph{a subset of a subgroup $H$ of $G$ is precompact in $H$ if and only if it is precompact in~$G$}.

A subset $A$ of a TVS $E$ is \emph{bounded} if, for every neighborhood $O$ of 
zero in $E$, one can find an  $n\in \omega$ such that $A\subset n O$.
Every precompact subset of a TVS is bounded~\cite[Lemma 1.5.5]{BSS2017tvs}. Therefore, 
every functionally bounded subset of a TVS is bounded.

A subset $A$ of a TVS $E$ is \emph{weakly bounded} if, for every
continuous linear functional $\varphi\in E'$, the set $\varphi(A)$ is bounded in $\R$, i.e., if $A$ is 
bounded in~$E_w$.
Clearly, any bounded set is weakly bounded. 
If $E$ is a locally convex space, 
then $A\subset E$ is bounded if and only if $A$ is weakly bounded \cite[Theorem~1.11.15]{BSS2017tvs}. 
Moreover, \emph{a set $A\subset E$ is weakly bounded if and only if it is precompact in $E_w$}. Indeed, 
the map $x\mapsto (\varphi(x))_{\varphi\in E'}$ is a topological isomorphism of $E_w$ onto a subgroup of 
 the topological product $\R^{E'}$, and a subset of  $\R^{E'}$ is precompact if and 
only if all of its projections  on the factors are bounded.

A topological group $G$ is said to be \emph{locally precompact} if the identity element of $G$ has a 
precompact neighborhood. A topological group is locally precompact if and only if its 
Ra\u\i kov completion is locally compact (see \cite[3.7.I]{AT}). 

We need the following modification of the $\Delta$-system lemma; in fact, we use it only 
for countable families of sets.

\begin{proposition}[{see \cite[Problems~1 and~3 on pp.~107, 421,~422]{Komjath-Totik}}]
\label{prop:Delta-system}
Suppose given a set $X$, a positive integer $n$, and an infinite family  $\A$ of subsets of $X$ 
such that $|M|=n$ for all $M\in \A$ and the cardinal $\gamma=|\A|$ is regular.
Then there exists an $R\subset X$ and a $\B\subset \A$ such that $|\B|=\gamma$ and $A\cap B=R$ for any 
distinct $A,B\in \B$.
\end{proposition}

\begin{corollary}
\label{cor:Delta-system}
Given a set $X$, a positive integer  $n$, and an infinite family $\A$ of subsets of $X$ 
such that $|M|\leq n$ for all $M\in \A$, 
there exists an $R\subset X$ and an infinite family $\B\subset \A$ of sets of the same 
cardinality such that $A\cap B=R$ for any distinct $A,B\in \B$.
\end{corollary}

\begin{proof}
There exists an $m\leq n$ and a countably infinite set $\A'\subset\A$ such that $|A|=m$ for $A\in\A'$. 
Applying Proposition~\ref{prop:Delta-system} to $\A'$, we obtain the required family.
\end{proof}

\begin{remark}
\label{rem:delta}
In Corollary~\ref{cor:Delta-system}, we have $A\setminus R\neq\es$ for all $A\in\B$. 
Indeed, if $A=R$ for some $A\in \B$, then $R\subsetneq B$ for any other $B\in \B$, and hence $|B|>|A|$.
\end{remark}

In the  rest of this section
we collect several auxiliary facts about  $\varkappa$-Fr\'echet--Urysohn spaces and  
families absorbing convergent sequences; their complete proofs are provided in 
the companion paper~\cite{LRS-groups}.
For the reader's convenience,  we include  sketches of the proofs.

A topological space $X$ is \emph{strongly $\varkappa$-Fr\'{e}chet--Urysohn} if, given any $x\in X$ 
and any decreasing sequence $(U_n)_{n\in \omega}$ of open subsets of $X$ for which 
$x\in \bigcap_{n\in \omega} \cl U_n$, there exists a sequence $(x_n)_{n \in\omega}$ converging to $x$ 
such that $x_n \in  U_n$ for all $n \in \omega$. This is the $\varkappa$-analogue of the strong 
Fr\'echet--Urysohn property~\cite{Siwiec}, and the following theorem is  an analogue of P.~Nyikos' 
theorem~\cite{Nyikos1981}.

\begin{theorem}[{\cite[Theorem~2.1]{GOR2026kFUgr}}]
\label{t1}
Any $\varkappa$-Fr\'{e}chet--Urysohn topological group is strongly $\varkappa$-Fr\'{e}chet--Urysohn.
\end{theorem}

\begin{proposition}[{\cite{LRS-groups}}]
\label{p1}
If $X$ is strongly $\varkappa$-Fr\'{e}chet--Urysohn, $x\in X$, and $(U_n)_{n\in\omega}$ is a \textup(not 
necessarily decreasing\textup) sequence of open subsets of $X$ for which $x\in \bigcap_{n\in \omega} \cl U_n$, 
then there exist a sequence $(x_k)_{k\in\omega}$ converging to $x$ and a strictly increasing sequence 
$(n_k)_{k\in\omega}$ of nonnegative integers such that $x_k\in U_{n_k}$ for all $k\in\omega$.
\end{proposition}

\begin{proof}
Let $W_n=\bigcup_{k\ge n} U_k$. The sequence of open sets $(W_n)_{n\in\omega}$ is decreasing, and $x\in\cl W_n$. Hence there is a sequence 
$(y_n)_{n\in\omega}$ converging to $x$  such that $y_n \in W_n$, i.e., $y_n\in U_{m_n}$ for some $m_n\geq n$. Choose $j_0<j_1<\dots$ so that 
$m_{j_0}<m_{j_1}<\dots$\,. Then $x_k=y_{j_k}$ and $n_k=m_{j_k}$  are as required.
\end{proof}

\begin{definition}
\label{def-absorb}
We say that a family $\mathscr F$ of subsets of $X$ \emph{absorbs convergent sequences} in $X$ if, for 
every convergent sequence $S=(x_n)_{n\in \omega}$ in $X$, there exists an 
$F\in \mathscr F$ such that $S \subset F$.
\end{definition}

\begin{remark}
\label{rem:finer}
If a family $\mathscr F$ absorbs convergent sequences in $(X,\tau)$, then it absorbs convergent sequences in $(X,\tau')$ for any 
topology $\tau'$ finer than~$\tau$ and $\{F\cap Y:F\in\mathscr F\}$ absorbs convergent sequences in any $Y\subset X$.
\end{remark}

\begin{lemma}[{\cite{LRS-groups}}]
\label{lemma:ind-limit}
If a space $Z$ is the inductive limit of an increasing sequence $(K_n)_{n\in\omega}$ of its closed subspaces, 
then the family $\{K_n: n\in\omega\}$ absorbs convergent sequences in $Z$.
\end{lemma}

\begin{proof}
Suppose that $x_k\to x$ but $\{x_k:k\in\omega\}$ is contained in no $K_n$. Since every finite set is contained in 
some $K_n$, we can choose $k_0<k_1<\dots$ so that $x_{k_j}\notin K_j$. Let $x\in K_m$ and 
$A=\{x_{k_j}:j\geq m\}$. Then $x\in \cl A\setminus A$. On the other hand, $A\cap K_n$ 
is finite for each $n$, so that $A$ is 
closed in $Z$. This contradiction proves the lemma.
\end{proof}

\begin{proposition}[{\cite{LRS-groups}}]
\label{prop:kfu-group-as}
Suppose that an increasing family $\{F_n: n\in \omega\}$ of closed subsets of a 
$\varkappa$-Fr\'{e}chet--Urysohn topological group $G$ absorbs convergent sequences in~$G$. 
Then $F_n$ is a neighborhood of the identity element $e$ of $G$ for some $n \in \omega$.
\end{proposition}

\begin{proof}
Suppose not. Then $e\in \cl{G\setminus F_n}$ for all $n\in\omega$, and the sequence of open sets $(G\setminus F_n)_{n\in\omega}$ is decreasing. By 
Theorem~\ref{t1} there exists a sequence $(x_n)_{n\in\omega}$ converging to $e$ with $x_n\notin F_n$. 
This sequence is contained in no $F_n$, which is
a contradiction.
\end{proof}

\begin{proposition}[{\cite{LRS-groups}}]
\label{prop:kfu-group-loc-precomp}
Suppose that a family $\{F_n: n\in \omega\}$ of closed precompact subsets of a topological group $G$ absorbs convergent 
sequences in~$G$. Then any $\varkappa$-Fr\'{e}chet--Urysohn subgroup of $G$ is locally precompact. 
\end{proposition}

\begin{proof}
Let $H$ be a $\varkappa$-Fr\'{e}chet--Urysohn subgroup of $G$. The sets $P_n=(F_0\cup\dots\cup F_n)\cap H$ are closed in $H$ and precompact in $G$ (and hence in $H$).
Also, the family $\{P_n: n\in\omega\}$ is increasing and absorbs convergent sequences in $H$. 
By Proposition~\ref{prop:kfu-group-as} one of the sets $P_n$ is a neighborhood of the identity in~$H$.
\end{proof}

Recall that a space $Z$ is \emph{hemicompact} if there are compact sets $K_n\subset Z$, $n\in\omega$, such that 
every compact subset of $Z$ is contained in some $K_n$; we then say that $Z$ is hemicompact \emph{with respect to} 
$\{K_n:n\in\omega\}$. Replacing $K_n$ by $K_0\cup\dots\cup K_n$, we may assume that the sequence of sets $K_n$ is increasing.

\begin{corollary}
\label{cor:hemicompact-group}
Let $G$ be a topological group which is hemicompact with respect to a family $\{K_n:n\in\omega\}$. Then every 
$\varkappa$-Fr\'echet--Urysohn subgroup $H$ of $G$ is locally precompact. More precisely, $K_n\cap H$ is a 
neighborhood of the identity in $H$ for some $n$, provided that the sequence of sets $K_n$ is increasing.
\end{corollary}

\begin{proof}
Every compact subset of a topological group is closed and precompact. If a sequence $(x_k)_{k\in\omega}$ converges to $x$ 
in $G$, then $\{x\}\cup\{x_k: k\in\omega\}$ is compact and hence contained in some $K_n$; thus, the family 
$\{K_n: n\in\omega\}$ absorbs convergent sequences in $G$. It remains to apply Proposition~\ref{prop:kfu-group-loc-precomp} (and, for 
the second assertion, Proposition~\ref{prop:kfu-group-as} to the increasing family $\{K_n\cap H: n\in\omega\}$, which 
absorbs convergent sequences in $H$ by Remark~\ref{rem:finer}).
\end{proof}

\section{Locally finite-dimensional subgroups\\ of topological vector spaces}
\label{section3}

\begin{definition}
\label{lfd}
Let $V$ be a TVS. We say that its subgroup $G$ is \emph{locally finite-dimensional} if there exists 
a finite-dimensional subspace $E$ of $V$ for which $G\cap E$ is open in~$G$. 
\end{definition}

Equivalently, $G$ is locally finite-dimensional if it has an open subgroup contained in a finite-dimensional linear subspace of~$V$.

\begin{remark}
\label{rem1}
 Clearly, any subgroup of a locally finite-dimensional group is locally finite-dimensional. Note also that 
if $G$ is a locally finite-dimensional subgroup of a TVS $V$, then so is its closure 
$\cl G$. Indeed, let $E$ be a finite-dimensional subspace of $V$ for which $E\cap G$ is open in $G$. 
The subspace $E$ is closed in $V$~\cite[Corollary~1.5.2]{BSS2017tvs}. Take an open set $O\subset V$ such that 
$O\cap G=E\cap G$. Since $G$ is dense in $\cl G$, we have 
$\es\ne O\cap \cl G\subset \cl{O\cap G}=\cl{E\cap G}\subset E\cap \cl G$. 
Thus, the subgroup $E\cap \cl G$ of $\cl G$ has 
nonempty interior and hence is open  in $\cl G$ \cite[Corollary~1.3.3]{AT}.
\end{remark}

\begin{lemma}\label{lemma:fd-subgroup}
Let $V$ be a TVS, and let $G$ be a subgroup of $V$ topologically isomorphic to a 
finite-dimensional vector space $W$. Then $G$ is a finite-dimensional vector subspace of~$V$.
\end{lemma}

\begin{proof}
Let $\varphi\colon W\to G$ be a topological isomorphism. We claim that any continuous homomorphism 
$\psi\colon \R\to V$ is linear, i.e., $\psi(t)=t\psi(1)$ for all $t\in \R$. Indeed, for $m\in\mathbb Z$ and 
 $k\in \omega\setminus\{0\}$, we have $k\psi(m/k)=\psi(m)=m\psi(1)$, whence 
$\psi(m/k)=\frac mk\psi(1)$ (we divide in the vector space $V$). Thus, the continuous maps $t\mapsto \psi(t)$ and 
$t\mapsto t\psi(1)$ coincide on $\Q$ and hence on $\R$, because $V$ is Hausdorff. 

Applying this to the continuous homomorphisms $t\mapsto \varphi(tw)$, $w\in W$, we obtain 
$\varphi(tw)=t\varphi(w)$ for all $t\in \R$ and $w\in W$. Hence $\varphi$ is an injective linear map, 
and $G=\varphi(W)$ is a linear subspace of $V$ of dimension $\dim W$.
\end{proof}

Given an abelian topological group $G$ and its subgroups $E$ and $H$ such that the map 
$E\times H\to G$ defined by $(x,y)\mapsto x+y$ is a topological isomorphism, we write $G=E\oplus 
H$.

\begin{theorem}\label{th:lc-subgroups}
For  a TVS $V$ and a  subgroup $G$ of $V$, 
the following conditions are equivalent:
\begin{enumerate}
\item[\rm (i)]
$G$ is locally compact;
\item[\rm (ii)]
$G$ is locally compact and metrizable;
\item[\rm (iii)]
$G$ is a Lie group;
\item[\rm (iv)]
$G$ is a closed locally finite-dimensional  subgroup of $V$;
\item[\rm (v)]
there exists a finite-dimensional subspace $E$ of $V$ such that $E\subset G$ and $E$ is open in~$G$;
\item[\rm (vi)]
there exists a finite-dimensional subspace $E$ of $V$ and a discrete subgroup $H$ of $G$ closed in $V$
for which $G=E\oplus H$.
\end{enumerate}
\end{theorem}

\begin{proof}
By F.~Riesz' theorem (see \cite[Theorem~1.5.1]{BSS2017tvs}) any subspace $E$ of finite dimension $n$ in any 
TVS $V$ is topologically isomorphic to $\mathbb R^n$. Therefore, each 
of the conditions (iv)--(vi) implies (i); in the case of (iv), note that $E\cap G$ is closed in $E$ and 
hence locally compact. Clearly, (ii) and (iii) imply (i) as well. 

Suppose that condition (i) holds. Then $G$ is closed in $V$ \cite[Proposition~1.4.19]{AT}. By the vector group splitting theorem  
\cite[Theorem~7.57]{Hofmann-Morris} any locally compact abelian group $G$ can be written as the direct 
product of $E=\R^m$ (where $m$ is a nonnegative integer) and a locally compact group $H$ 
containing a compact open subgroup. By Lemma~\ref{lemma:fd-subgroup}\, $E$ is a finite-dimensional subspace of $V$. 
Any compact subgroup $K$ of $V$ is trivial: if $g\in K\setminus\{0\}$, then $K\cap \R g$ is a compact 
subgroup of the line $\R g\cong \R$ containing $g\ne 0$, which is impossible. Hence $H$ is discrete, and $H$ is 
closed in $V$, being locally compact. 
 Therefore, (i) implies (ii)--(vi).
\end{proof}

\begin{theorem}\label{th:lfd-subgroups}
For  a TVS $V$ and a  subgroup $G$ of $V$, 
the following conditions are equivalent:
\begin{enumerate}
\item[\rm (i)]
$G$ is locally finite-dimensional;
\item[\rm (ii)]
the closure of $G$ in $V$ is locally compact;
\item[\rm (iii)]
$G$ is locally precompact.
\end{enumerate}
\end{theorem}
\begin{proof}
According to Remark~\ref{rem1}, 
a subgroup $G$ of a TVS $V$ is locally finite-dimensional in $V$ if and only if 
so is its closure $\cl G$. Therefore, Theorem~\ref{th:lc-subgroups} implies the equivalence of 
(i)  and (ii). Clearly, (iii)  follows from~(ii).

Let us prove  that (iii) implies (i). Let $\widehat V$ be the completion of $V$, and let $H$ be the closure of $G$ 
in $\widehat V$. Being a closed subgroup of the complete group $\widehat V$, the group $H$ is complete, and $G$ 
is dense in $H$; hence $H$ is the Ra\u\i kov completion of $G$. Since $G$ is locally precompact, $H$ is locally 
compact. By Theorem~\ref{th:lc-subgroups} applied to the TVS $\widehat V$, there exists a 
finite-dimensional subspace $E$ of $\widehat V$ for which $E\cap H$ is open in $H$. Then $E\cap G$ is open in $G$, 
and $E_0=E\cap V$ is a finite-dimensional subspace of $V$ such that $E_0\cap G=E\cap G$. Thus, $G$ is locally 
finite-dimensional.
\end{proof}

\begin{proposition}
\label{prop-new}
A linear subspace $W$ of a TVS $V$ is locally finite-dimensional as a subgroup of $V$ if and only if $W$ is finite-dimensional. 
\end{proposition}

\begin{proof}                                                    		
If a subspace $W$ is locally finite-dimensional, then it contains an open finite-dimensional subspace and is 
therefore locally compact, because any finite-dimensional space is topologically isomorphic to $\mathbb R^n$ for 
some $n\in \omega$. It is well known that any locally compact vector space is 
finite-dimensional~\cite[Theorem~1.5.6]{BSS2017tvs}. 

The converse is obvious.
\end{proof}

\section{A criterion for local finite dimensionality}
\label{section4}

Given a TVS  $V$, a norm $\norm$ on $V$ (not necessarily related to the topology 
of $V$), and an $r > 0$,  we set 
\[
B(r)=\{x\in V:\|x\|\leq r\}. 
\]
We say that a set $A\subset V$ is \emph{norm-bounded} if there exists an $r\in \mathbb R$ 
for which $A\subset B(r)$. The sets $B(r)$, $r > 0$, form a neighborhood base at zero for a vector space 
topology on $V$, which may coincide or not with the given topology of $V$. 
We refer to this topology on $V$, as well as to the topology induced by it on any $A\subset V$, 
as the \emph{topology generated by the norm $\norm$}. 

\begin{proposition}
\label{prop:sg-tvs}
Let $V$ be a TVS, and let $\norm$ be a norm on $V$.
If $G$ is a subgroup of $V$ and some nonempty open set $U\subset G$ is norm-bounded, then
\begin{enumerate}
\item[{\rm(i)}]
the topology of $G$ is finer than that generated by $\norm$\textup;
\item[{\rm(ii)}]
if the topology on $V$ generated by $\norm$ is finer than the given topology of $V$, 
then the topology of $G$ coincides with  that generated by~$\norm$\textup;
\item[{\rm(iii)}]
if there exists an $r>0$ for which the ball $B(r)$ is precompact in $V$, then $G$ is locally finite-dimensional.
\end{enumerate}
\end{proposition}

\begin{proof}
(i) Suppose that $U\subset B(r)$, where $r>0$, and $x\in U$. Let $W=U-x$. Then $W$ is a neighborhood 
of zero in $G$ and $W\subset B(2r)$.
For each  positive integer $m$, the continuity of addition in $G$ 
implies the existence of a neighborhood $W_m$ of zero in $G$ such that $m W_m\subset W$ and hence 
$W_m\subset B(2r/m)$. Since both topologies are invariant under translations, 
the topology of $G$ is finer than that generated by the norm~$\norm$.

Assertion (ii) follows from (i) because the topology of $G$ is induced from $V$ and 
is therefore coarser than the norm topology.

(iii) As shown in the proof of (i), $G$ has a neighborhood $W$ of zero with $W\subset B(s)$ for some 
$s>0$. Let $m$ be a positive integer such that $s/m\le r$, and let $W'$ be a neighborhood of zero in $G$ 
for which $m W'\subset W$. Then $W'\subset B(r)$. By assumption $B(r)$ is precompact in $V$; hence so is 
$W'$, and $W'$ is precompact in $G$ as well (see Section~\ref{Preliminaries}). 
Consequently, $G$ is locally precompact. By Theorem~\ref{th:lfd-subgroups} the group $G$ 
is locally finite-dimensional.
\end{proof}

\begin{proposition}
\label{prop:sg-baire-tvs}
Let $V$ be a TVS, and let $\norm$ be a norm on $V$ such that
 the ball $B(r)=\{x\in V:\|x\|\leq r\}$ is closed in $V$ for any $r \geq 0$.
If $G$ is a Baire subgroup of $V$, then
\begin{enumerate}
\item[{\rm(i)}]
the topology of $G$ is finer than that generated by $\norm$\textup;
\item[{\rm(ii)}]
if the topology generated by $\norm$ on $V$ is finer than the given topology of $V$, 
then the topology of $G$ coincides with that generated by $\norm$\textup;
\item[{\rm(iii)}]
if there exists an $r>0$ for which the ball $B(r)$ is precompact, then $G$ 
is locally finite-dimensional.
\end{enumerate}
\end{proposition}

\begin{proof}
In view of Proposition~\ref{prop:sg-tvs}, it suffices to show that there exists a non\-empty 
norm-bounded set $U \subset G$ open in $G$.

Let $F_n = G \cap B(n)$ for $n \in \omega$. Then all $F_n$ are closed in $G$ and 
$G = \bigcup_{n \in \omega} F_n$, so that $G$  being Baire  implies  the 
existence of an $n\in \omega$ for which the interior 
$\Int(B(n)\cap G)$ of $B(n)\cap G$ in $G$ is nonempty. Clearly, $U=\Int(B(n)\cap G)$ 
has the required properties. 
\end{proof}

\begin{proposition} 
\label{prop:kfu-group-tvs}
Let $V$ be a TVS, and let $G$ be a $\varkappa$-Fr\'{e}chet--Urysohn 
subgroup of $V$. Suppose that $\norm$ is a norm on $V$ satisfying the following conditions:
\begin{enumerate}
\item[{\rm(a)}]
for any $r > 0$, the ball $B(r)$ is closed in $V$;
\item[{\rm(b)}]
the family $\mathscr F=\{B(n): n\in \omega\}$ absorbs convergent sequences in $V$. 
\end{enumerate}
Then
\begin{enumerate}
\item[{\rm(i)}]
the topology of $G$ is finer than that generated by $\norm$\textup;
\item[{\rm(ii)}]
if the topology generated by $\norm$ on $V$ is finer than the given topology of $V$, 
then the topology of $G$ coincides with that generated by $\norm$\textup;
\item[{\rm(iii)}]
if there exists an $r>0$ for which $B(r)$ is precompact, then $G$ is locally finite-dimensional.
\end{enumerate}
\end{proposition} 

\begin{proof}
In view of Proposition~\ref{prop:sg-tvs}, it suffices to prove the existence of a non\-empty norm-bounded 
set $U \subset G$ open in~$G$.

Let $F_n=G\cap B(n)$ for $n\in\omega$.
Then the increasing family $\{F_n: n\in \omega\}$ of closed subsets of the $\varkappa$-Fr\'{e}chet--Urysohn topological group $G$ absorbs convergent sequences in $G$.
According to Proposition~\ref{prop:kfu-group-as}, the interior $Int(F_n)$ of $F_n$ in $G$ is a neighborhood 
of zero in $G$ for some $n \in \omega$; clearly, it is norm-bounded.
\end{proof}

The following theorem summarizes Propositions~\ref{prop:sg-baire-tvs} and~\ref{prop:kfu-group-tvs}; we use it as the main tool to establish our results on TVS.

\begin{theorem}
\label{th:criterion}
Let $V$ be a TVS, and let $\norm$ be a norm on $V$ such that
\begin{enumerate}
\item[{\rm(a)}] every ball $B(r)$, $r>0$, is closed in $V$;
\item[{\rm(b)}] the family $\{B(n): n\in \omega\}$ absorbs convergent sequences in $V$;
\item[{\rm(c)}] the ball $B(r)$ is precompact in $V$ for some \textup(equivalently, for all\textup) $r>0$.
\end{enumerate}
Then every $\varkappa$-Fr\'echet--Urysohn subgroup and every Baire subgroup of $V$ is locally finite-dimensional, 
and every $\varkappa$-Fr\'echet--Urysohn or Baire linear subspace of $V$ is finite-dimensional. The same holds for 
every linear subspace of~$V$ with the induced topology.
\end{theorem}

\begin{proof}
The first assertion follows from Propositions~\ref{prop:kfu-group-tvs}(iii) and~\ref{prop:sg-baire-tvs}(iii), and 
the second assertion follows from the first assertion and Proposition~\ref{prop-new}. If $W$ is a linear subspace of $V$ and a subgroup 
$G\subset W$ is locally finite-dimensional in $V$, i.e., $E\cap G$ is open in $G$ for a finite-dimensional 
subspace $E$ of $V$, then $E\cap W$ is a finite-dimensional subspace of $W$ and $(E\cap W)\cap G=E\cap G$.
\end{proof}

\section{Normed spaces with the weak and weak$^*$ topologies}
\label{section5}

In this section, $E$ is a normed space and $B(r)=\{x\in E:\|x\|\le r\}$ is the closed ball of radius $r>0$. 
The completeness of $E$ is not assumed.

\begin{lemma}\label{lemma:bs-as}
Given any normed space $E$, the family $\mathscr B=\{B(n): n\in \omega\}$ consists of closed precompact 
subsets of $E_w$ and absorbs convergent sequences in~$E_w$.
\end{lemma}

\begin{proof}
For any $A\subset E$, the following conditions are equivalent:
\begin{enumerate}
\item[\rm (a)] $A$ is norm-bounded;
\item[\rm (b)] $A$ is weakly bounded, i.e., bounded in $E_w$;
\item[\rm (c)] $A$ is precompact in $E_w$.
\end{enumerate}
The implication (a)$\implies$(b) is obvious, and (b)$\Longleftrightarrow$(c) was noted earlier in Section \ref{Preliminaries}. Let us prove 
(b)$\implies$(a). The dual $E'$ of $E$ is a Banach space, and the canonical embedding $E\to E''$ is an 
isometry. If $A$ is weakly bounded, then the family of functionals $\{\varphi\mapsto\varphi(x): x\in A\}$ on 
the Banach space $E'$ is pointwise bounded; by the Banach--Steinhaus uniform boundedness 
principle~\cite[Theorem~3.5.4]{BSS2017tvs}, it is norm-bounded, i.e., $A$ is norm-bounded.

By the Hahn--Banach theorem, 
$$
B(n)=\bigcap\{\{x:|\varphi(x)|\le n\}:\varphi\in E',\ \|\varphi\|\le 1\},
$$ 
so that 
$B(n)$ is closed in $E_w$. Consequently, the family $\mathscr B$ consists of closed precompact subsets of $E_w$. 
Moreover, any convergent sequence in $E_w$ is bounded in $E_w$ and hence norm-bounded, which implies that the 
family $\mathscr B$ absorbs convergent sequences in~$E_w$.
\end{proof}

\begin{proposition}
\label{prop:bs-kfu-baire-sgs}
Let $E$ be a normed space, and let $\tau$ be a vector space topology on $E$ between the topologies of $E$ and $E_w$.
 If $G$ is a $\varkappa$-Fr\'{e}chet--Urysohn or Baire subgroup of the space $(E,\tau)$, then the topology of $G$ 
coincides with the topology induced from~$E$.
\end{proposition}

\begin{proof}
By Lemma~\ref{lemma:bs-as} and Remark~\ref{rem:finer} the family $\mathscr B=\{B(n): n\in \omega\}$ consists of closed 
subsets of $(E,\tau)$ and absorbs convergent sequences in $(E,\tau)$.

If $G$ is $\varkappa$-Fr\'{e}chet--Urysohn, then the 
topology of $G$ coincides with that induced from $E$ by Proposition~\ref{prop:kfu-group-tvs}(ii).

If $G$ is Baire, then  the topology of $G$ coincides 
with that induced from $E$ by Proposition~\ref{prop:sg-baire-tvs}(ii).
\end{proof}

\begin{theorem}
\label{th:bs-kfu-baire-sgs}
Let $E$ be a normed space, and let $G$ be a subgroup of the space $E_w$.
If $G$ is $\varkappa$-Fr\'{e}chet--Urysohn or Baire, then $G$ is locally finite-dimensional.
\end{theorem}

\begin{proof}
According to Lemma~\ref{lemma:bs-as}, the family $\mathscr B=\{B(n): n\in \omega\}$ consists of closed 
precompact subsets of $E_w$ and absorbs convergent sequences in $E_w$.

If $G$ is 
$\varkappa$-Fr\'{e}chet--Urysohn, then it is locally precompact by 
Proposition~\ref{prop:kfu-group-loc-precomp} and hence locally finite-dimensional by 
Theorem~\ref{th:lfd-subgroups}. If $G$ is Baire, then it is locally finite-dimensional 
by Proposition~\ref{prop:sg-baire-tvs}(iii). 

Alternatively, the theorem follows directly from Theorem~\ref{th:criterion}, because, by 
Lemma~\ref{lemma:bs-as}, the norm of $E$ satisfies conditions (a)--(c) on $E_w$.
\end{proof}

Theorem~\ref{th:bs-kfu-baire-sgs} and Proposition~\ref{prop-new} immediately imply the 
characterization of $\varkappa$-Fr\'{e}chet--Urysohn or Baire linear subspaces of $E_w$.

\begin{corollary}
\label{cor:bs-kfu-baire-sgs}
Let $E$ be a normed space, and let $L$ be a linear subspace of $E_w$.
If $L$ is $\varkappa$-Fr\'{e}chet--Urysohn or Baire, then $L$ is finite-dimensional.
\end{corollary}

\begin{remark}
\label{rem:Lw}
For an infinite-dimensional Banach space $E$, the fact that $E_w$ is not Baire was observed in \cite{Edgar1984}.
Also, the fact that $E_w$ is not Fr\'echet--Urysohn follows from  results of \cite[p.~280]{SW1991}.
\end{remark}

If $V$ is a locally convex space with a continuous norm $p$, then the $p$-balls are weakly closed and absorb  
weakly convergent sequences. However, the $p$-balls need not be weakly precompact, and Theorem~\ref{th:criterion} 
does not apply.  The following  question is open.

\begin{question}
\label{q:cont-norm}
Let $V$ be a locally convex space admitting a continuous norm. Is every $\varkappa$-Fr\'echet--Urysohn 
\textup(Fr\'echet--Urysohn, metrizable\textup) subgroup of $V_w$ locally finite-dimensional? Is every 
$\varkappa$-Fr\'echet--Urysohn linear subspace of $V_w$ finite-dimensional?
\end{question}

Next we turn our attention to  duals of Banach spaces  endowed with the weak$^*$ topology.
For a Banach space $E$, we consider $E'$ with the dual norm $\|\varphi\|=\sup\{|\varphi(x)|:\|x\|\le1\}$.
By $B(r)=\{\varphi\in E':\|\varphi\|\le r\}$ we denote the closed ball of radius $r$ in $E'$.

\begin{proposition}
\label{lemma:weakstar}
Let $E$ be a Banach space. Then
\begin{enumerate}
\item[\rm (i)] every ball $B(n)$ is compact in $E'_{w^*}$;
\item[\rm (ii)] every compact subset of $E'_{w^*}$ is norm-bounded and hence  contained in some $B(n)$.
\end{enumerate}
Thus, $E'_{w^*}$ is hemicompact with respect to the family $\{B(n): n\in\omega\}$. In particular, this family 
absorbs convergent sequences in~$E'_{w^*}$.
\end{proposition}

\begin{proof}
(i) This is the classical Banach--Alaoglu theorem.

(ii) Let $K\subset E'$ be weak$^*$ compact. For every $x\in E$, the map $\varphi\mapsto\varphi(x)$ is continuous 
on $E'_{w^*}$; hence the set $\{\varphi(x):\varphi\in K\}$ is compact and therefore bounded. Since $E$ is complete, 
the Banach--Steinhaus uniform boundedness principle~\cite[Theorem~3.5.4]{BSS2017tvs} applied to the family $K$ 
of continuous linear functionals on~$E$ shows that $\sup\{\|\varphi\|:\varphi\in K\}<\infty$. 

The last assertion holds because a convergent sequence together with its limit is compact.
\end{proof}

\begin{remark}
\label{rem:hemicompact-ref}
Apparently, the hemicompactness of $E'_{w^*}$ is folklore; it is stated without proof in \cite[Example~1.3(c)]{KampoukosMercourakis2013}.
As we have mentioned before, the fact that, for any infinite-dimensional Banach space $E$, the spaces $E_w$ and $E'_{w^*}$ are
not $k$-spaces is well known \cite{KM2017}.
\end{remark}

\begin{theorem}
\label{th:weakstar}
Let $E$ be a Banach space. Then every $\varkappa$-Fr\'echet--Urysohn or Baire subgroup of $E'_{w^*}$ is locally 
finite-dimensional, and every $\varkappa$-Fr\'echet--Urysohn or Baire linear subspace of $E'_{w^*}$ is finite-dimensional. 
\end{theorem}

\begin{proof}
Let $G$ be a $\varkappa$-Fr\'echet--Urysohn subgroup of $E'_{w^*}$. By Proposition~\ref{lemma:weakstar} and 
Corollary~\ref{cor:hemicompact-group}, the set $G\cap B(n)$ is a neighborhood of zero in $G$ for some $n$. This 
set is precompact in $G$, because $B(n)$ is compact and hence precompact in $E'_{w^*}$ (see 
Section~\ref{Preliminaries}). Therefore, $G$ is locally precompact and, by Theorem~\ref{th:lfd-subgroups}, 
locally finite-dimensional.

If $G$ is a Baire subgroup of $E'_{w^*}$, then $G$ is locally finite-dimensional by 
Proposition~\ref{prop:sg-baire-tvs}(iii), because the balls $B(r)$ are compact and hence closed and precompact 
in $E'_{w^*}$.

Alternatively, the theorem follows from Theorem~\ref{th:criterion}, because, by 
Proposition~\ref{lemma:weakstar}, the dual norm satisfies conditions (a)--(c) of that theorem.
\end{proof}

The completeness of $E$ in Theorem~\ref{th:weakstar} cannot be dropped.

\begin{example}
\label{ex:incomplete}
Let $E$ be any normed space of countably infinite \textup(Hamel\textup) dimension. If 
 $\{e_k:k\in\omega\}$ is a Hamel 
basis of $E$, then the weak$^*$ topology is generated by the countably many functionals 
$\varphi\mapsto\varphi(e_k)$, and hence $E'_{w^*}$ is an infinite-dimensional metrizable TVS.
\end{example}

\begin{remark}
\label{rem:normed-dual}
For an arbitrary normed space $E$, the balls $B(n)$ are still compact in $E'_{w^*}$ by the Banach--Alaoglu theorem, 
which holds for normed spaces; thus, conditions (a) and (c) of Theorem~\ref{th:criterion} are satisfied. Therefore, 
by Proposition~\ref{prop:sg-baire-tvs}(iii), \emph{every Baire subgroup of $E'_{w^*}$ is locally finite-dimensional 
for any normed space $E$}; only the $\varkappa$-Fr\'echet--Urysohn part of Theorem~\ref{th:weakstar} requires 
condition~(b). In Example~\ref{ex:incomplete}, the space $E'_{w^*}=\bigcup_n B(n)$ is indeed not Baire.
\end{remark}

\section{Convergent sequences in vector spaces\\ with various topologies}\label{section6}

Recall that $L_b(X)$ denotes $E(X)$ endowed with the topology generated by the functionals 
$\chi\mapsto\chi(f)$, $f\in C_b(X)$ (equivalently, with the topology induced from $L_p(\beta X)$, where $\beta X$ is the Stone--\v{C}ech compactification of $X$). 
By $\norm_\infty$ we denote the $\sup$-norm on~$C_b(X)$.

\begin{proposition}
\label{prop:lbx-bounded}
For any space $X$, the following assertions hold.
\begin{enumerate}
\item[\rm (i)]
$\|\chi\|=\sup\{|\chi(f)|: f\in C_b(X),\ \|f\|_\infty\le 1\}$ for every $\chi\in E(X)$.
\item[\rm (ii)]
For any $r\ge 0$, the ball $B_r(X)=\{\chi\in E(X):\|\chi\|\le r\}$ is closed in $L_b(X)$.
\item[\rm (iii)]
A set $B\subset L_b(X)$ is bounded in $L_b(X)$ if and only if there exists an $n\in \omega$ 
for which $B\subset B_n(X)$. 
\end{enumerate}
\end{proposition}

\begin{proof}
(i) Clearly, $|\chi(f)|\le\|\chi\|\,\|f\|_\infty$. Conversely, let $\chi=\sum_{i=1}^n\lambda_ix_i$, where the 
$x_i$ are distinct and $\lambda_i\ne0$. Take pairwise disjoint open sets $O_i\ni x_i$ and continuous functions 
$g_i\colon X\to[0,1]$  such that $g_i(x_i)=1$ and $g_i=0$ outside $O_i$. Then 
$f=\sum_{i=1}^n\operatorname{sign}(\lambda_i)\,g_i\in C_b(X)$, $\|f\|_\infty\le1$, and $\chi(f)=\|\chi\|$.

(ii) By (i), the function $\chi\mapsto\|\chi\|$ is the supremum of a family of continuous functions on 
$L_b(X)$ and hence is lower semicontinuous.

(iii) By (i), the map $J\colon E(X)\to C_b(X)'$  defined by $J(\chi)(f)=\chi(f)$ 
is a linear isometry of $(E(X),\norm)$ into the dual of the Banach space $(C_b(X),\norm_\infty)$.
If $B$ is bounded in $L_b(X)$, then the set $\{\chi(f): \chi\in B\}$ is
bounded for every $f\in C_b(X)$. It follows from the Banach--Steinhaus uniform boundedness principle
\cite[Theorem~3.5.4]{BSS2017tvs} that $J(B)$ is norm-bounded in $C_b(X)'$, i.e., 
$B\subset B_n(X)$ for some $n$. 
Conversely, if $B\subset B_n(X)$, then $|\chi(f)|\le n\|f\|_\infty$ for all $\chi\in B$ and $f\in C_b(X)$; 
hence $B$ is bounded in the locally convex space $L_b(X)$.
\end{proof}

\begin{proposition}
\label{p6}
For any space $X$, the family of $B_n(X)$, $n\in \omega$, consists of closed subsets of $L_b(X)$ and 
absorbs convergent sequences in $L_b(X)$ \textup(and therefore in $L_p(X)$, $L(X)$, $V(X)$, and $E(X)$ with any 
vector space topology finer than that of $L_b(X)$\textup).
\end{proposition}

\begin{proof}
This follows from Proposition~\ref{prop:lbx-bounded}, because any convergent sequence is bounded, and 
from Remark~\ref{rem:finer}.
\end{proof}

\begin{lemma}
\label{lemma:En-closed}
For any space $X$ and any $n\in\omega$, the set $E_n(X)$ is closed in $L_b(X)$ \textup(and hence in 
$L_p(X)$, $L(X)$, and $V(X)$\textup).
\end{lemma}

\begin{proof}
Let $\chi\in E(X)\setminus E_n(X)$, and let $x_0,\dots,x_n$ be distinct points of $\supp\chi$. 
Take pairwise disjoint open sets $O_i\ni x_i$ such that $O_i\cap\supp\chi=\{x_i\}$ and continuous functions 
$g_i\colon X\to[0,1]$ with $g_i(x_i)=1$ and $g_i=0$ outside $O_i$, $i=0,\dots,n$. Then $\chi(g_i)\ne0$ for 
all $i$ and the set $W=\{\eta\in E(X):\eta(g_i)\ne0,\ i=0,\dots,n\}$ is an open neighborhood of $\chi$ in 
$L_b(X)$. If $\eta\in W$, then $\supp\eta$ meets each of the $n+1$ disjoint sets $O_i$, so that 
$|\supp\eta|>n$. Thus, $W\cap E_n(X)=\es$.
\end{proof}

Applying  Proposition \ref{prop:kfu-group-tvs}(i) (its conditions~(a) and~(b) hold by 
Proposition~\ref{p6}), we obtain the following corollary. 

\begin{corollary}
\label{c3}
Let $X$ be any topological space, and let $V$ be $E(X)$ endowed with any vector space topology finer than that 
of $L_b(X)$. Then the topology of any $\varkappa$-Fr\'echet--Urysohn subgroup of $V$ is finer than that 
induced by the norm~$\norm$. 
\end{corollary}

\begin{proposition}
\label{p7}
For any topological space $X$, the family $\{E_n(X): n\in \omega\}$ absorbs convergent sequences in $V(X)$.
\end{proposition}
\begin{proof}
Let $\xi=(\chi_n)_{n\in\omega}$ be a convergent sequence in $V(X)$, and let $V(X,\beta X)$ denote the subspace 
of $V(\beta X)$ generated by $X$; clearly, $V(X,\beta X)$ and $V(X)$ coincide as sets, the topology of $V(X)$ 
is finer than that of $V(X,\beta X)$, and $\xi$ remains a convergent 
sequence in $V(X,\beta X)$.
We set 
$$
F_n=E_n(\beta X)\cap B_n(\beta X)=\Bigl\{\sum_{i=1}^n\lambda_i x_i: x_i\in \beta X, 
\lambda_i\in \mathbb R, \sum_{i=1}^n|\lambda_i|\le n\Bigr\}
$$
for $n\in \om$ (some of the $\lambda_i$ may vanish). 
Note that each $F_n$ is a continuous image of the space 
$$
(\beta X)^n\times K,\quad\text{where }K=\{(\lambda_1, \dots, \lambda_n)\in \mathbb R^n: |\lambda_1|+\dots+|\lambda_n|\le n\}.
$$
Since $K$ is a closed subset of the compact space $[-n,n]^n$, it follows that every $F_n$ is compact. 
Note also that $F_{n}\subset F_{n+1}$ and 
$V(\beta X)$ is the inductive limit of the sequence $(F_n)_{n\in\om}$ 
\cite[Theorem~3.1]{GM2017ftvs}. By Lemma~\ref{lemma:ind-limit} and Remark~\ref{rem:finer}, 
$\xi\subset F_n\cap E(X)\subset E_n(X)$ for some~$n\in\omega$.
\end{proof}

\section{Subgroups of free topological vector spaces}\label{section7}

\begin{lemma}\label{lemma:lfd}
Let $\tau$ be any vector space topology on $E(X)$. Then a subgroup $G$ of $(E(X), \tau)$ is locally 
finite-dimensional if and only if $G\cap E_n(X)$ has nonempty interior in $G$ for some $n\in \omega$. 
\end{lemma}

\begin{proof}
$(\Rightarrow)$
Suppose that $G$ is locally finite-dimensional. Then $G\cap L$ is open in $G$ for some finite-dimensional 
subspace $L\subset E(X)$. We have  $L\subset E(M)$ for a finite $M\subset X$. Let $n=|M|$. 
Then $G\cap L\subset E_{n}(X)$ and, consequently, $G\cap E_n(X)$ has nonempty interior in~$G$.

$(\Leftarrow)$
Suppose that $G\cap E_n(X)$ has nonempty interior in $G$ for some $n\in \omega$.
If $g+O\subset E_n(X)$, where $g\in G$ and $O$ is a neighborhood of zero in $G$, then 
$O\subset E_n(X)-g\subset E_{2n}(X)$. Thus, $G\cap E_{2n}(X)$ is a neighborhood of zero in $G$. Let $U\subset G$ be a neighborhood of zero of $G$ such 
that $U+\dots+U\subset E_{2n}(X)$ ($2n+1$ summands).

The subgroup $\langle U\rangle$ is open in $G$, and $\langle U\rangle\subset E(S)$, where $S=\supp U$. Therefore, it suffices to 
prove that $S$ is finite: then $E(S)$ is a finite-dimensional subspace and $G\cap E(S)\supset\langle U\rangle$ is open 
in~$G$. Suppose that $S$ is infinite. 
Then the following claim holds.

\smallskip
\noindent
{\bf Claim 1.} 
{\it
 There exist $\chi_i \in U$ and $x_i\in X$, $i=1,\dots,2n+1$, such that
$x_i\in \supp \chi_i \setminus \bigcup_{j\ne i} \supp\chi_j$ for each~$i$. 
}

\begin{proof}[Proof of Claim 1]
Since $S$ is infinite and all supports are finite, the family $\{\supp \chi:\chi\in U\}$ is 
infinite; each of its members  has at most $2n$ elements.  
Corollary~\ref{cor:Delta-system} and Remark~\ref{rem:delta}  imply the 
existence of a $Y\subset X$ and a family $\{\chi_i:i\in \om\}\subset U$ such that the sequence $\{\supp\chi_i:i\in \om\}$ consists of distinct 
elements, $\supp\chi_i\setminus Y\ne\es$, and $\supp\chi_i\cap \supp\chi_j=Y$ for distinct $i,j\in \om$. We take 
$x_i\in \supp \chi_i \setminus Y$ for $i=1,\dots,2n+1$. 
\end{proof}

We set 
\[
\chi = \sum_{i=1}^{2n+1} \chi_i.
\]
Note that $\{x_1, \dots, x_{2n+1}\}\subset \supp \chi$. Indeed, 
suppose that $x_i\notin \supp \chi$. Then $\chi\in E(X\setminus \{x_i\})$. Since 
$x_i\notin \supp \chi_j$ for $j\ne i$, we also have $\chi_j\in E(X\setminus \{x_i\})$ for 
$j\ne i$. Therefore, 
$$
\chi_i=\chi-\sum_{j\ne i}\chi_j \in E(X\setminus \{x_i\}), 
$$
which contradicts the fact that $x_i\in \supp \chi_i$. 

Thus, $|\supp \chi|>2n$ and, consequently, $\chi\notin E_{2n}(X)$. On the other hand, 
$\chi\in U+\dots+U \subset E_{2n}(X)$. This contradiction proves the lemma. 
\end{proof}

\begin{theorem}\label{th:bsg-lfd}
Let $\tau$ be a vector space topology on $E(X)$ finer than the topology of $L_b(X)$. Then any Baire subgroup 
$G$ of $(E(X), \tau)$ is locally finite-dimensional. 
\end{theorem}

\begin{proof}
The sets $E_n(X)$ are closed in $L_b(X)$ (by Lemma~\ref{lemma:En-closed}) and hence in $(E(X), \tau)$.
Since $G$ is Baire and $G=\bigcup_{n\in\omega} E_n(X)\cap G$, it follows that the interior of $E_n(X)\cap G$ 
in $G$ is nonempty for some~$n$.
Lemma~\ref{lemma:lfd} implies that $G$ is locally finite-dimensional.
\end{proof}

\begin{theorem}
\label{th:B}
Any Baire subgroup of any of the spaces $L_b(X)$, $L_p(X)$, $L(X)$, and $V(X)$ is locally finite-dimensional.
\end{theorem}

Proposition~\ref{prop-new} immediately implies  a similar description of Baire linear subspaces. 

\begin{corollary}
\label{cor:B}
Any Baire linear subspace of any of the spaces $L_b(X)$, $L_p(X)$, $L(X)$, and $V(X)$ is finite-dimensional.
\end{corollary}

The following crucial lemma can be deduced from the characterization of bounded subsets of $L_p(X)$ given in 
\cite[Proposition 2.1]{BG2022flvs}. That characterization, in turn, follows from the characterization of 
bounded subsets of $L(X)$ given in \cite[Proposition 4.10]{G2019mapg}: since $L_p(X)$ is $L(X)$ with the weak 
topology, the bounded subsets of $L(X)$ and $L_p(X)$ coincide~\cite[Theorem 1.11.15]{BSS2017tvs}.

For the reader's convenience, we present a direct proof of the lemma.

\begin{lemma}
\label{lemma:lpx-bounded}
For any space $X$ and any bounded set $B\subset L_p(X)$, 
the set $\supp B=\bigcup_{\chi\in B} \supp \chi$ is functionally bounded 
in~$X$. 
\end{lemma}

\begin{proof}
Suppose that $\supp B$ is not functionally bounded. Then there exists a discrete family $\{U_n:n\in\om\}$ of open sets such that 
$U_n\cap \supp B\neq \es$ for $n\in \om$. For each $n\in\om$, take $\chi_n\in B$ such that $\supp \chi_n\cap 
U_n\neq\es$. Since every $\supp\chi_n$ is finite and meets only finitely many members of the discrete family 
$\{U_n\}$, there exists a subsequence $(\chi_{n_k})_{k\in\om}$ such that 
$U_{n_k} \cap \supp \chi_{n_i}= \es$ for $i<k<\om$; clearly, the $\chi_{n_k}$ are pairwise 
distinct.

For each $k\in \om$, we set $\xi_k=\chi_{n_k}$ and choose $x_k\in \supp \xi_k$ and a neighborhood $V_k$ of $x_k$ 
so that $x_k\in V_k\subset U_{n_k}$ and $\{x_k\}=V_k\cap \supp \xi_k$. 
Let $f_k\colon X\to [0,1]$ be a continuous function such that $f_k(x_k)=1$ and $f_k^{-1}((0,1])\subset V_k$. 
Then $\xi_k(f_k)\neq 0$ and $V_k\cap \supp \xi_i=\es$ for $i<k$; hence $\xi_i(f_k)=0$ for $i<k$.

We define $\lambda_k\in \R$ by induction on $k \in \om$ as follows. 
We set $\lambda_0=1$. Assuming that $k>0$ and $\lambda_0, \dots, \lambda_{k-1}$ are already defined, 
we consider $c=\xi_k\bigl(\sum_{i=0}^{k-1}\lambda_i f_i\bigr)$ and choose $\lambda_k\in \R$ so that 
$\lambda_k\xi_k(f_k)+c>k$; it exists because $\xi_k(f_k)\neq 0$. 

We have 
\[
\xi_k\biggl(\sum_{i=0}^{k}\lambda_i f_i\biggr)>k.
\]
Let $f=\sum_{i=0}^{\infty}\lambda_i f_i$. Since the family $\{V_i\}$ is discrete, this sum is locally 
finite; therefore, the function $f$ is continuous. Since $\xi_k(f_i)=0$ for $i>k$, we have 
$\xi_k(f)=\xi_k\bigl(\sum_{i=0}^{k}\lambda_i f_i\bigr)>k$ for $k\in\om$.
This contradicts the boundedness of the set~$B$.
\end{proof}

\begin{theorem}
\label{t4}
Let $\tau$ be a vector space topology on $E(X)$ between the topologies of $L_p(X)$ and $L(X)$. Then any  
$\varkappa$-Fr\'echet--Urysohn subgroup of $(E(X), \tau)$ is locally finite-dimensional. 
\end{theorem}
\begin{proof}
Let $G$ be a $\varkappa$-Fr\'echet--Urysohn subgroup of $E=(E(X), \tau)$.
From Corollary \ref{c3} it follows that the topology of $G$ is finer than the topology generated 
by~$\norm$.

\smallskip
\noindent
{\bf Claim 2.}
{\it
If $\xi=(v_n)_n\subset G$ is a sequence convergent in $G$, then the set $\supp \xi=\bigcup_{n\in\omega} \supp v_n$ is 
finite.}
 
\begin{proof}[Proof of Claim 2] 
Let $v$ be the limit of the sequence $\xi$.  We can assume that $v=0$, because 
$\supp v_n\subset \supp(v_n-v)\cup\supp v$ and $(v_n-v)_n$ is a 
sequence in $G$ converging to zero.
Since $\xi$ is a convergent sequence in $L_p(X)$ and hence bounded, it follows by Lemma~\ref{lemma:lpx-bounded} that 
$M=\supp \xi$ is a functionally bounded subset 
of $X$. Let $H=\langle \xi \rangle$ (the subgroup of $G$ generated by $\xi$; note that $H\subset E(M)$), 
and let $\widetilde H$ be the closure of $H$ in the completion 
$\widehat E$ of $E$.

We show that $\widetilde H$ is locally compact.
Since $M$ is functionally bounded in $X$ and hence in $L(X)$, it follows that $M$ is precompact 
in $L(X)$ (see, e.g., \cite[Lemma~2.9]{Gab_Schur}). Therefore, so is  its absolutely convex hull 
$D=E(M)\cap B_1(X)$ \cite[Proposition~1.8.2]{BSS2017tvs}.
Since the topology of $L(X)$ is finer than  that of $E$, the closure 
$\widetilde D$ of $D$ in $\widehat E$ is compact. Note that $B_1(X)\cap G$ is a neighborhood of zero in $G$, 
since  the topology of $G$ is finer than the norm topology generated by $\norm$. 
Therefore, $\widetilde D\cap \widetilde H$ is a compact neighborhood of zero in 
$\widetilde H$, because $B_1(X)\cap H$ is a neighborhood 
of zero in $H$, its closure in $\widetilde H$ is a neighborhood of zero in $\widetilde H$ ($H$ is dense 
in $\widetilde H$), and $B_1(X)\cap H\subset D$. Thus, $\widetilde H$ is locally compact.

Theorem~\ref{th:lc-subgroups} (applied  to the TVS $\widehat E$) implies that 
$\widetilde H$ is locally finite-di\-men\-sion\-al. Consequently, $H$ is locally finite-dimensional 
(by Remark~\ref{rem1}; see also the end of the proof of Theorem~\ref{th:lfd-subgroups}). Hence there exists a 
finite $S\subset X$ such that $H\cap E(S)$ is a neighborhood of zero in $H$.
Since $v_i\to0$ in $H$, there exists an $n\in\omega$ for which $\supp v_i\subset S$ for $i>n$.
Therefore, $M\subset S\cup \bigcup_{i\leq n}\supp v_i$ and hence $M$ is finite.
This completes the proof of Claim~2.
\end{proof}

\smallskip

Suppose that $G$ is not locally finite-dimensional.
Then, for any finite $M\subset X$, the set $G\cap E(M)$ is closed (because $E(M)$ is 
finite-dimensional) and nowhere dense in $G$ (otherwise, the subgroup $G\cap E(M)$ would be open in 
$G$). Since $G$ 
is $\varkappa$-Fr\'echet--Urysohn and $0\in\cl{G\setminus E(M)}$, there exists a sequence 
$\xi_M\subset G\setminus E(M)$ converging to zero. Note that $\supp\xi_M$ is finite by Claim~2.

Let us construct by induction on $n$ an increasing sequence $(M_n)_n$ of finite subsets of $X$ 
and a sequence $(\xi_n)_n$ of sequences converging to zero satisfying the following conditions:
\begin{enumerate}
\item[(a)] $M_0=\es$;
\item[(b)] $\supp \xi_i\subset M_n$ for $i<n$;
\item[(c)] $\xi_n\subset G\setminus E(M_n)$.
\end{enumerate}

\emph{Base case}.\enspace We set $M_0=\es$ and $\xi_0=\xi_\es$.

\emph{Induction step}.\enspace  For $n>0$, we set $M_n=M_{n-1}\cup \supp \xi_{n-1}$ and $\xi_n=\xi_{M_n}$. 

Suppose that $\xi_n=(v_{ni})_i$. For $i\in\om$, let $r_{ni}>0$ be the minimum of the absolute values of 
the (nonzero) coefficients of $v_{ni}$, and let 
$U_{ni}=\{v\in G:\|v-v_{ni}\|<r_{ni}\}$. By Corollary~\ref{c3} the set $U_{ni}$ is open in $G$. 
If $v\in U_{ni}$, then $\supp v_{ni}\subset \supp v$: otherwise, for $s\in \supp v_{ni}\setminus\supp v$, 
the coefficient of $s$ in $v-v_{ni}$ has absolute value at least $r_{ni}$. Let $U_n=\bigcup_{i\in \om}U_{ni}$.
Then $0\in \cl{U_n}$, because $v_{ni}\to0$ as $i\to\infty$. Moreover, 
$\supp u \cap (M_{n+1}\setminus M_n) \neq \es$ for $u\in U_n$, because 
$\supp v_{ni}\subset M_{n+1}$ by (b) and $\supp v_{ni}\not\subset M_n$ by~(c).

Theorem \ref{t1} implies that $G$ is strongly $\varkappa$-Fr\'echet--Urysohn.
By Proposition~\ref{p1} there exist a sequence $(u_k)_k$ converging to zero and a strictly increasing 
sequence $(n_k)_k$ of nonnegative integers such that $u_k\in U_{n_k}$. Since $\supp u_k \cap (M_{n_k+1}\setminus M_{n_k}) \neq \es$ 
for $k\in\om$ and the sets $M_{n+1}\setminus M_n$ are pairwise disjoint, 
it follows that the set $\supp (u_k)_k$ is infinite in contradiction to Claim~2.
This completes the proof.
\end{proof}

\begin{corollary}\label{cor:lbx-kfu}
For any topological space $X$, every $\varkappa$-Fr\'echet--Urysohn subgroup of $L_b(X)$ is locally finite-dimensional.  
\end{corollary}

\begin{proof}
Consider the map $J$ from $(E(X),\norm)$ to the dual of the Banach space $(C_b(X),\norm_\infty)$
defined by the rule $J(\chi) = (f\mapsto\chi(f))$.
By Proposition~\ref{prop:lbx-bounded} $J$ is a linear isometry, and, by definition, the topology of $L_b(X)$ 
is induced by the weak$^*$ topology of $C_b(X)'$. Thus, $L_b(X)$ is a linear subspace of $(C_b(X))'_{w^*}$, and it 
remains to apply Theorem~\ref{th:weakstar}. (Alternatively, one can apply Theorem~\ref{t4} to $\beta X$, because 
$L_b(X)$ is a linear subspace of $L_p(\beta X)$.)
\end{proof}

\begin{theorem}\label{th:VX-kfu}
For any topological space $X$, every $\varkappa$-Fr\'echet--Urysohn subgroup of $V(X)$ is locally finite-dimensional.  
\end{theorem}
\begin{proof}
It follows from Proposition~\ref{p7} that the family $\{E_n(X): n\in \omega\}$ absorbs convergent sequences in $V(X)$.
Let $G$ be a $\varkappa$-Fr\'echet--Urysohn subgroup of $V(X)$. 
Then the family  $\{F_n=G\cap E_n(X):n\in \omega\}$ absorbs convergent sequences in $G$.
The sets $F_n$ are closed in $G$ (by Lemma~\ref{lemma:En-closed}) and the sequence of $F_n$ is increasing.
It follows from Proposition~\ref{prop:kfu-group-as} that $G\cap E_n(X)$ is a neighborhood of zero in $G$ for some $n\in \omega$.
Lemma \ref{lemma:lfd} implies that $G$ is locally finite-dimensional.
\end{proof}

Theorems \ref{th:bsg-lfd}, \ref{t4}, and \ref{th:VX-kfu}, Corollary~\ref{cor:lbx-kfu}, and the fact that 
any locally finite-dimensional linear subspace of $E(X)$ is finite-dimensional (by Proposition~\ref{prop-new})
imply the following two theorems.

\begin{theorem}\label{th:lss-baire}
Let $\tau$ be a vector space topology on $E(X)$ finer than the topology of $L_b(X)$. Then any Baire linear subspace of $(E(X), \tau)$ is finite-dimensional.
\end{theorem}

\begin{theorem}\label{th:lss-kfu}
Let $\tau$ be a vector space topology on $E(X)$. Suppose that one of the following conditions holds:
\begin{enumerate}
\item[{\rm(i)}]
$\tau$ is between the topologies of $L_p(X)$ and $L(X)$\textup;
\item[{\rm(ii)}]
$(E(X),\tau)=V(X)$\textup;
\item[{\rm(iii)}]
$(E(X),\tau)=L_b(X)$.
\end{enumerate}
Then any $\varkappa$-Fr\'echet--Urysohn linear subspace of $(E(X),\tau)$ is finite-di\-men\-sional.
\end{theorem}

Below we provide an example which shows that  in Theorems~\ref{t4} and 
 \ref{th:lss-kfu}(i), unlike in Theorems~\ref{th:bsg-lfd} and~\ref{th:lss-baire}, 
the assumption that the topology under consideration is coarser than that of $L(X)$
cannot be dropped.

\begin{example}
\label{ex:norm}
Let $X$ be an infinite compact space. Recall that for every $\chi=\sum_{i=1}^n\lambda_i x_i \in E(X)$ the norm $\|\chi\|=\sum_{i=1}^n |\lambda_i|$
has been defined. Let $\tau_1$ be the topology on $E(X)$ generated by this norm 
$\norm$. Since $|\chi(f)|\le\|\chi\|\,\|f\|_\infty$ and all continuous functions on $X$ are bounded, $\tau_1$ is 
finer than the topology of $L_p(X)=L_b(X)$. The space $(E(X),\tau_1)$ is metrizable and infinite-dimensional, 
so that it is a Fr\'echet--Urysohn (and hence $\varkappa$-Fr\'echet--Urysohn) linear subspace of itself which is 
not finite-dimensional. Note that $\tau_1$ is not coarser than the topology of $L(X)$, because 
$\|x-y\|=2$ for distinct $x,y\in X$ and $X$ is not discrete. (Of course, $(E(X),\tau_1)$ is not Baire by virtue of Theorem~\ref{th:lss-baire}.)
\end{example}

\begin{question}
\label{q5}
Let $\tau$ be a vector space topology on $E(X)$ between the topologies of $L_p(X)$ and $V(X)$. Is it true that 
every $\varkappa$-Fr\'echet--Urysohn subgroup of $(E(X),\tau)$ is locally finite-dimensional? 
\end{question}

\end{document}